\documentclass[12pt]{amsart}
\usepackage{graphicx}
\usepackage{amssymb}
\usepackage{amsmath}
\usepackage{amscd}
\usepackage{upgreek}
\usepackage{mathrsfs}
\usepackage{stmaryrd}
\usepackage{longtable}
\usepackage{txfonts}
\usepackage[T1]{fontenc}
\usepackage{eucal}
\usepackage{wrapfig}
\usepackage{float}
\usepackage{xypic}
\usepackage{dsfont}
\usepackage{xcolor}
\usepackage{color}
\usepackage[colorlinks,linkcolor=blue,anchorcolor=blue,
citecolor=blue]{hyperref}
\usepackage{hyperref}
\usepackage[hmarginratio=1:1,top=32mm,columnsep=20pt]{geometry}
\DeclareMathOperator{\Hom}{Hom}	
\allowdisplaybreaks	

\newtheorem{theorem}{Theorem}[section]
\newtheorem{prop}[theorem]{Proposition}
\newtheorem{defn}[theorem]{Definition}

\newtheorem{rem}[theorem]{Remark}

\newtheorem{exam}[theorem]{Example}
\newtheorem{thm}[theorem]{Theorem}

\newcommand{\C}{\mathbb{C}} 
\newcommand{\N}{\mathbb{N}} 
\newcommand{\F}{\mathbb{F}} 
\newcommand{\g}{\mathfrak{g}} 
\newcommand{\gl}{\mathfrak{gl}} 

\newcommand{\LA}{\Longleftarrow} 
\newcommand{\RA}{\Longrightarrow} 
 
\newcommand{\ra}{\longrightarrow} 
  
\newcommand{\GL}{{\rm GL}} 
\newcommand{\SL}{{\rm SL}} 
\newcommand{\om}{\omegaup} 
\newcommand{\Der}{{\rm Der}}  
\newcommand{\Aut}{{\rm Aut}}  

\begin{document}
\setlength{\oddsidemargin}{0cm}
\setlength{\evensidemargin}{0cm}

\title{\scshape Derivations, Automorphisms, and Representations of Complex $\omegaup$-Lie Algebras}
\author{\scshape Yin Chen}
\address{School of Mathematics and Statistics, Northeast Normal University, Changchun 130024, P.R. China}
\email{ychen@nenu.edu.cn}

\author{\scshape Ziping Zhang}
\address{School of Mathematics and Statistics, Northeast Normal University, Changchun 130024, P.R. China}
\email{zhangzp586@nenu.edu.cn}

\author{\scshape Runxuan Zhang}
\address{School of Mathematics and Statistics, Northeast Normal University, Changchun 130024, P.R. China}
\email{zhangrx728@nenu.edu.cn}

\author{\scshape Rushu Zhuang}
\address{School of Mathematics and Statistics, Northeast Normal University, Changchun 130024, P.R. China}
\email{zhuangrs246@nenu.edu.cn}

\date{\today}
\def\shorttitle{Derivations, Automorphisms, and Representations of Complex $\omegaup$-Lie Algebras}

\begin{abstract} Let $(\g,\om)$ be a finite-dimensional non-Lie complex $\om$-Lie algebra. 
We study the derivation algebra $\Der(\g)$ and the automorphism group $\Aut(\g)$ of $(\g,\om)$. 
We introduce the notions of 
$\om$-derivations and $\om$-automorphisms of $(\g,\om)$ which naturally preserve the bilinear form $\om$.
We show that the set  $\Der_{\om}(\g)$ of all $\om$-derivations is a Lie subalgebra of $\Der(\g)$ and 
the set  $\Aut_{\om}(\g)$ of all $\om$-automorphisms is a subgroup of $\Aut(\g)$.
 For any 3-dimensional and 4-dimensional nontrivial $\om$-Lie algebra $\g$, we compute $\Der(\g)$ and $\Aut(\g)$ explicitly, and  study some Lie group properties of  $\Aut(\g)$. We also study representation theory of $\om$-Lie algebras. We show that all 3-dimensional nontrivial  $\om$-Lie algebras are multiplicative, as well as we provide a 4-dimensional example of $\om$-Lie algebra that is not multiplicative. Finally, we show that any irreducible representation of the simple $\om$-Lie algebra $C_{\alphaup}(\alphaup\neq 0,-1)$  is 1-dimensional. 
\end{abstract}


\maketitle
\baselineskip=17pt


\section{\scshape Introduction}

\setcounter{equation}{0}
\renewcommand{\theequation}
{1.\arabic{equation}}
\setcounter{theorem}{0}
\renewcommand{\thetheorem}
{1.\arabic{theorem}}

In 2007, Nurowski \cite{Nur2007} was motivated by  the study of isoparametric hypersufaces in Riemannian geometry (Bobie\'{n}ski-Nurowski \cite{BN2007} and Nurowski \cite{Nur2008}), and 
introduced the notion of $\om$-Lie algebras which can be viewed as a kind of natural generalization of Lie 
algebras. Specifically, 
for a finite-dimensional  vector space $\g$ over a field $\F$ of characteristic zero equipped with a skew symmetric bracket $[-,-]:\g\times \g \longrightarrow \g$  and a bilinear form $\om: \g\times \g \longrightarrow \F$,  we say that the triple $(\g,[-,-],\om)$ is an \textit{$\om$-Lie algebra} if
\begin{equation}\label{Jacob-identity}
[[x,y],z]+ [[y,z],x]+  [[z,x],y]=\om(x,y)z+\om(y,z)x+\om(z,x)y
\end{equation}
for all $x,y,z\in \g$. Equation (\ref{Jacob-identity}) is called the $\om$-\textit{Jacobi identity}. By this identity, we see that the bilinear form $\om$ is skew-symmetric for any $\om$-Lie algebra $(\g,\om)$.
Clearly,  an $\om$-Lie algebra $(\g,\om)$ with $\om=0$ is an ordinary  Lie algebra, which  is called a \textit{trivial $\om$-Lie algebra.}  An $\om$-Lie algebra $(\g,\om)$ is  called \textit{nontrivial} (or \textit{non-Lie}) if $\g$ is not a Lie algebra.

In 2010, Zusmanovich \cite{Zus2010} developed a fundamental structure theorem on finite-dimensional $\om$-Lie algebras over an algebraically closed field of characteristic zero, showing that nontrivial finite-dimensional $\om$-Lie algebras are either low-dimensional or have a very degenerate structure, see  \cite[Section 9, Theorem 1]{Zus2010}. 
Two important and interesting results were derived:  
If $\g$ is a finite-dimensional $\om$-Lie algebra with non-degenerate $\om$, then $\dim (\g)=2$ (see \cite[Lemma 8.1]{Zus2010});
A finite-dimensional semisimple $\om$-Lie algebra is either a Lie algebra, or has dimension $\leqslant 4$ (see \cite[Theorem 2]{Zus2010}). These results indicate essentially the importance of low-dimensional $\om$-Lie algebras.

By the definition of $\om$-Lie algebras, there are no nontrivial $\om$-Lie algebras in the cases of dimensions 1 and 2.
The first example of  nontrivial 3-dimensional $\om$-Lie algebra was given by Nurowski \cite{Nur2007} in which the author 
 gave a classification of 3-dimensional $\om$-Lie algebras over the field of real numbers, under 
the action of 3-dimensional orthogonal group.
In 2014, we  extended a method for classifying 3-dimensional complex Lie algebras appeared in Fulton-Harris \cite[Sections 10.2-10.4]{FH1991}
to obtain a classification of 3-dimensional complex $\om$-Lie algebras,
see Chen-Liu-Zhang \cite[Theorem 2]{CLZ2014}. This  classification, together with a result of Zusmanovich
\cite[Lemma 8.2]{Zus2010} which asserts that any  4-dimensional $\om$-Lie algebra over an algebraically closed field contains a 3-dimensional subalgebra, has led us to complete a classification of 4-dimensional complex Lie algebras
in a recent paper Chen-Zhang \cite[Theorem 1.5]{CZ2016}. 
In particular, we obtained a classification of nontrivial finite-dimensional complex simple $\om$-Lie algebras, see \cite[Theorem 1.7]{CZ2016}.

The first purpose of  this paper is to study derivations and automorphisms of  low-dimensional complex $\om$-Lie algebras, which are two kinds of classical objects in the study of nonassociative algebra and representation theory. 
In what follows we fix the ground field to be $\C$, the field of complex numbers. Let's recall the definitions of classical  derivation and automorphism of any nonassociative algebra. 
Let $\g$ be a finite-dimensional $\om$-Lie algebra (as a nonassociative algebra) and $d:\g\ra\g$ be a linear map.
We say that $d$ is a \textit{derivation} of $\g$ if 
$$d([x,y])=[d(x),y]+[x,d(y)]$$ for any $x,y\in\g$.
We write $\gl(\g)$ for the general linear Lie algebra on $\g$. Then the set $\Der(\g)$ of all derivations of $\g$ forms a Lie subalgebra of $\gl(\g)$, which is called the \textit{derivation algebra} of $\g$. A linear isomorphism  $\rhoup: \g\ra\g$ is called an \textit{automorphism} of $\g$ if $$\rhoup([x,y])=[\rhoup(x),\rhoup(y)]$$ for any $x,y\in\g$. The set $\Aut(\g)$ of all automorphisms of $\g$ forms a closed  Lie subgroup of the general linear group $\GL(\g)$, which means that $\Aut(\g)$ is a matrix Lie group,  see 
Sagle-Walde \cite[Proposition 7.1]{SW1973}. We call $\Aut(\g)$ the \textit{automorphism group} of $\g$. 
It is well-known that the Lie algebra of $\Aut(\g)$ is just the derivation algebra $\Der(\g)$, see 
Sagle-Walde \cite[Proposition 7.3 (b)]{SW1973}.
However,  $\om$-Lie algebras as a kind of special nonassociative algebras and a generalization of Lie algebras, deserve to have their feature and properties. Comparing with  Hom-Lie algebras and their derivations (automorphisms), see Sheng \cite[Section 3]{She2012} and Jin-Li \cite[Definition 1.2]{JL2008}, we introduce the following two notions: $\om$-derivation and $\om$-automorphism, which both preserve naturally the skew-symmetric bilinear form $\om$. 

\begin{defn}{\rm
Let $\g$ be a finite-dimensional $\om$-Lie algebra.
A derivation $d\in\Der(\g)$ is called an \textit{$\om$-derivation} of $\g$ if $$\om(d(x),y)+\om(x,d(y))=0$$ for any $x,y\in\g$.
An automorphism $\rhoup\in\Aut(\g)$ is called an \textit{$\om$-automorphism} of $\g$ if $$\om(x,y)=\om(\rhoup(x),\rhoup(y))$$ for any $x,y\in\g$.
}\end{defn}
\noindent We write $\Der_{\om}(\g)$ for the set consisting of all $\om$-derivations of $\g$, and $\Aut_{\om}(\g)$ for the set of all $\om$-automorphisms of $\g$. Clearly, $\Der_{\om}(\g)\subseteq \Der(\g)$ and $\Aut_{\om}(\g)\subseteq \Aut(\g)$. 
However, Proposition \ref{pro5.1} below indicates that the two equalities do not necessary hold here.
In this paper we calculate 
$$\Der(\g), \Der_{\om}(\g),\Aut(\g) \textrm{ and } \Aut_{\om}(\g)$$
for 3-dimensional and 4-dimensional $\om$-Lie algebras. We also study some Lie algebra properties of $\Der(\g)$ and  Lie group properties of $\Aut(\g)$ for 3-dimensional $\om$-Lie algebras.

The second purpose of this paper is to study representation theory of $\om$-Lie algebras.  Given an $\om$-Lie algebra $\g$ and a
$\g$-module $V$, we construct a  semi-direct product on $\g\oplus V$  such that 
 $(\g\oplus V,[-,-],\Omega)$ becomes an $\Omega$-Lie algebra. We clarify some assertions in 
Zusmanovich \cite[Section 2]{Zus2010} and give detailed proofs. We also study multiplicative $\om$-Lie algebras, showing that 
all 3-dimensional nontrivial $\om$-Lie algebras are multiplicative while giving an example of non-multiplicative $\om$-Lie algebra in 4-dimensional case. Furthermore, we study irreducible representations of $\om$-Lie algebras. Recall that a $\g$-module $V$ is called \textit{irreducible} if $V$ has no nontrivial submodules, i.e., if $V_{0}\subseteq V$ is any submodule of $V$, then $V_{0}$ is either $\{0\}$ or equal to $V$. We characterize all irreducible representations of $C_{\alphaup} (\alphaup\neq 0,-1)$ in detail, showing that there exist only 1-dimensional irreducible representations for $C_{\alphaup} (\alphaup\neq 0,-1)$.

This paper is organized as follows. Section 2 is a preliminary section which contains the classification lists of 3-dimensional and 
4-dimensional nontrivial complex  $\om$-Lie algebras.
We show that $\Der_{\om}(\g)\subseteq \Der(\g)$ is a Lie subalgebra and $\Aut_{\om}(\g)\leqslant \Aut(\g)$ is a subgroup
for any finite-dimensional $\om$-Lie algebra $\g$.
In Section 3 we calculate $\Der(\g)$ and $\Der_{\om}(\g)$, and study their structures, 
for any 3-dimensional  $\om$-Lie algebra $\g$.  We give a detailed proof for the case $\g=L_{1}$ (Proposition \ref{pro3.1}), omitting the proofs of other cases because of the similarity.
 Section 4 is devoted to calculating  
$\Aut_{\om}(\g)$ and $\Aut(\g)$ when $\g$ is a 3-dimensional $\om$-Lie algebra. Usually, the automorphism group of a nonassociative algebra is more complicated than the corresponding  derivation algebra. 
Here we study $\Aut(L_{1})$ and $\Aut(A_{\alphaup})$ in detail, see Propositions \ref{pro4.1} and \ref{pro4.3}. In particular,  we use a ``trick'' coming from Linear Algebra to characterize $\Aut(A_{\alphaup})$.
We prove that $\Aut(L_{1})$ and $\Aut(A_{\alphaup})$ are connected matrix Lie groups, and moreover $\exp(\Der(L_{1}))=\Aut(L_{1})$ and 
$\exp(\Der(A_{\alphaup}))=\Aut(A_{\alphaup})$. In Section 5, we apply similar arguments to calculate 
derivations and  automorphisms for 4-dimensional complex  $\om$-Lie algebras, without detailed proofs for saving space.
The main results on derivations and automorphisms are summarized in Tables \ref{3d}-\ref{4a}.
Section 6 is devoted to clarifying some discussions in Zusmanovich \cite[Section 2]{Zus2010}.
We give a sufficient and necessary condition (essentially due to Zusmanovich) for that an $\om$-Lie algebra has 1-dimensional module, which as an application
shows that all nontrivial 3-dimensional $\om$-Lie algebras are multiplicative (Proposition \ref{prop6.4}). 
In Section 7, we study irreducible representations of $C_{\alphaup}$ where $\alphaup\neq 0,-1$.
We show that any irreducible representation of $C_{\alpha}$ is 1-dimensional (Theorem \ref{thm7.1}). 
Note that if we extend the parameter $\alphaup$ to contain $-1$, then $C_{-1}$ is just isomorphic to the special linear Lie algebra
$\mathfrak{sl}_{2}(\C)$ which has a unique irreducible representation in any finite dimension. 
These results indicate that structure and representation of $\om$-Lie algebras 
might be very ``degenerate'' in sense of Zusmanovich \cite[Section 1]{Zus2010}. 

Throughout this paper, we assume that the ground field is the field of complex numbers $\C$, $W^{*}:=\Hom(W,\C)$ denotes the dual space of a vector space $W$, all vector spaces (modules) are finite-dimensional, and all $\om$-Lie algebras are nontrivial, unless stated otherwise. 

\section{\scshape Preliminaries}

\setcounter{equation}{0}
\renewcommand{\theequation}
{2.\arabic{equation}}
\setcounter{theorem}{0}
\renewcommand{\thetheorem}
{2.\arabic{theorem}}

Let's begin with the  classifications of  3-dimensional and 4-dimensional nontrivial $\om$-Lie algebras over $\C$.

\begin{thm} [Chen-Liu-Zhang \cite{CLZ2014}] \label{t3d}
Let $\g$ be a   3-dimensional nontrivial $\om$-Lie algebra over $\C$, then it must be isomorphic to one of the following algebras:
\begin{enumerate}
  \item \quad $L_{1}:\quad [x,z]=0,[y,z]=z, [x,y]=y\textrm{ and }\om(y,z)=\om(x,z)=0,\om(x,y)=1.$
  \item \quad $L_{2}:\quad [x,y]=0,[x,z]=y,[y,z]=z \textrm{ and } \om(x,y)=0, \om(x,z)=1, \om(y,z)=0.$
  \item \quad
  $A_{\alphaup} :\quad [x,y]=x,[x,z]=x+y, [y,z]=z+\alphaup x\textrm{ and  }
\om(x,y)=\om(x,z)=0,$
\begin{center}
$\om(y,z)=-1,\textrm{ where }\alphaup\in \mathbb{C}.$
\end{center}
  \item \quad $B: \quad [x,y]=y, [x,z]=y+z, [y,z]=x\textrm{ and } \om(x,y)=\om(x,z)=0,$
$\om(y,z)=2.$
   \item \quad $C_{\alphaup} : \quad [x,y]=y, [x,z]=\alphaup z, [y,z]=x\textrm{ and } \om(x,y)=\om(x,z)=0,$
  \begin{center}
$\om(y,z)=1+\alphaup,\textrm{ where } 0,-1\neq\alphaup\in \mathbb{C}.$
  \end{center}
\end{enumerate}
\end{thm}

\begin{thm} [Chen-Zhang \cite{CZ2016}]  \label{t4d}
Any  4-dimensional nontrivial $\om$-Lie algebra over $\C$ must be isomorphic to one of the following algebras:
$$\bigg\{L_{1,1}, \dots, L_{1,8}, L_{2,1}, L_{2,2}, L_{2,3}, L_{2,4}, \widetilde{B},  E_{1,\alphaup} (\alphaup\neq 0,1) ,F_{1,\alphaup} (\alphaup\neq 0,1) ,G_{1,\alphaup}, H_{1,\alphaup},\widetilde{A}_{\alphaup}, \widetilde{C}_{\alphaup}(\alphaup\neq 0,-1)\bigg\},$$
where the parameter $\alphaup\in\C$. For the nontrivial generating relations of these $\om$-Lie algebras see  \cite[Sections 2-6]{CZ2016} or Table \ref{4d} below.
\end{thm}

Let $\g$ be an $\om$-Lie algebra. The following two properties are elementary. 

\begin{prop}
$\Aut_{\om}(\g)$ is a subgroup of $\Aut(\g)$.
\end{prop}

\begin{proof}
Clearly, the identity map $I$ belongs to $\Aut_{\om}(\g)$. For any $\rhoup,\sigmaup\in\Aut_{\om}(\g)$, it suffices to show that 
$\rhoup^{-1}\cdot\sigmaup\in \Aut_{\om}(\g)$. Suppose $x,y\in\g$ are any two elements. Since $\om(x,y)=\om(I(x),I(y))=\om(\rhoup\cdot\rhoup^{-1}(x),\rhoup\cdot\rhoup^{-1}(y))=\om(\rhoup^{-1}(x),\rhoup^{-1}(y))$, we have $\rhoup^{-1}\in\Aut_{\om}(\g)$. Thus 
$\om(\rhoup^{-1}\cdot\sigmaup(x),\rhoup^{-1}\cdot\sigmaup(y))=\om(\sigmaup(x),\sigmaup(y))=\om(x,y)$ and so $\rhoup^{-1}\cdot\sigmaup\in \Aut_{\om}(\g)$.
\end{proof}

\begin{prop}
$\Der_{\om}(\g)$ is a Lie subalgebra of $\Der(\g)$.
\end{prop}

\begin{proof}
We need only to show that $[d,e]=d\cdot e-e\cdot d\in \Der_{\om}(\g)$ for any $d,e\in\Der_{\om}(\g)$.
Indeed, for any $x,y\in\g$, we have 
\begin{eqnarray*}
\om([d,e](x),y)&=&\om(d\cdot e(x)-e\cdot d(x),y)\\
&=&\om(d\cdot e(x),y)-\om(e\cdot d(x),y)\\
&=&-\om(e(x),d(y))+\om(d(x),e(y))
\end{eqnarray*}
and 
\begin{eqnarray*}
\om(x,[d,e](y))&=&\om(x,d\cdot e(y)-e\cdot d(y))\\
&=&\om(x,d\cdot e(y))-\om(x,e\cdot d(y))\\
&=&-\om(d(x),e(y))+\om(e(x),d(y)).
\end{eqnarray*}
Thus
$\om([d,e](x),y)+\om(x,[d,e](y))=0$, which implies that $[d,e]\in \Der_{\om}(\g)$.
\end{proof}

\section{\scshape Derivations in Dimension 3}

\setcounter{equation}{0}
\renewcommand{\theequation}
{3.\arabic{equation}}
\setcounter{theorem}{0}
\renewcommand{\thetheorem}
{3.\arabic{theorem}}

We first study the derivations of $L_{1}$, where $L_{1}$ has a basis $\{x,y,z\}$ and is defined as in Theorem \ref{t3d} (1). 
We denote by $E_{ij}$ the $n\times n$ matrix in which the $(i,j)$-entry is 1 and other entries are zero.  It is well-known that the
$\{E_{ij}\mid1\leqslant i,j\leqslant n\}$ is a basis for the general linear Lie algebra $\gl_{n}(\C)$.

\begin{prop}\label{pro3.1}
\begin{enumerate}
  \item $\Der(L_{1})$ is a 2-dimensional soluble (but not nilpotent) Lie algebra. 
  \item $\Der_{\om}(L_{1})=\Der(L_{1})$.
\end{enumerate}
\end{prop}

\begin{proof} (1)
For any $d\in\Der(L_{1})$, we assume that $d=(a_{ij})\in \gl_{3}(\C)$. Suppose $d(x)=a_{11}x+a_{12}y+a_{13}z, d(y)=a_{21}x+a_{22}y+a_{23}z,$ and $d(z)=a_{31}x+a_{32}y+a_{33}z.$
Since $[x,z]=0$ in $L_{1}$,  $0=d([x,z])=[d(x),z]+[x,d(z)]=[a_{11}x+a_{12}y+a_{13}z,z]+[x,a_{31}x+a_{32}y+a_{33}z]=a_{12}z+a_{32}y$. The last equality follows from the relations that $[y,z]=z$ and $[x,y]=y$ in $L_{1}$. Since $y,z$ are linearly independent, we have $$a_{12}=0=a_{32}.$$
Similarly, $a_{21}x+a_{22}y+a_{23}z=d(y)=d([x,y])=[d(x),y]+[x,d(y)]=[a_{11}x+a_{12}y+a_{13}z,y]+[x,a_{21}x+a_{22}y+a_{23}z]=(a_{11}+a_{22})y-a_{13}z$. Thus $a_{21}x-a_{11}y+(a_{23}+a_{13})z=0$. Since $x,y,z$ are linearly independent, 
$$a_{21}=0=a_{11} \textrm{ and }a_{23}=-a_{13}.$$
Finally, $a_{31}x+a_{32}y+a_{33}z=d(z)=d([y,z])=[d(y),z]+[y,d(z)]=[a_{21}x+a_{22}y+a_{23}z,z]+[y,a_{31}x+a_{32}y+a_{33}z]=
-a_{31}y+(a_{22}+a_{33})z$. Thus
$$a_{31}=0=a_{22}\textrm{ and } a_{31}=-a_{32}=0,$$ and 
$$d=\begin{pmatrix}
      0&  0 & a_{13}\\
      0& 0 &-a_{13}\\
      0&0&a_{33}
\end{pmatrix}$$
where $a_{13},a_{33}\in\C.$ Clearly, $\{E_{13}-E_{23}, E_{33}\}$ is a basis for $\Der(L_{1})$.
Therefore, $\Der(L_{1})$ is a 2-dimensional Lie algebra. 

Note that $[E_{13}-E_{23},E_{33}]=E_{13}-E_{23}$, so $\Der(L_{1})$ is not abelian. Since all 2-dimensional nonabelian Lie algebras are isomorphic to 
\begin{equation}
\label{g2}
\g_{2}: \{x,y\} \textrm{ is a basis with } [x,y]=y
\end{equation}
(see Fulton-Harri \cite[page 135]{FH1991}), we have $\Der(L_{1})\simeq \g_{2}$. 
Since $\g_{2}^{(1)}=[\g_{2},\g_{2}]$ is  the ideal generated by $y$ in $\g_{2}$, $\g_{2}^{(2)}=[\g_{2}^{(1)},\g_{2}^{(1)}]=0$ and so 
$\g_{2}$ is soluble. 
To see that $\g_{2}$ is not nilpotent, we assume by way of contradiction that there exists some 
$m\in\N^{+}$ such that $\g_{2}^{m+1}=[\g_{2}^{m},\g_{2}]=0$. However, since $[x,y]=y$, the element 
 $[\cdots[y,\underbrace{x],x],\cdots,x]}_{m}\in\g_{2}^{m+1}$ is not zero, contradiction.
 
 (2) For any $d\in\Der(L_{1})$, since $\om$ is bilinear, we need only to show that 
 \begin{equation}
\label{E3-2}
\om(d(\alphaup),\betaup)+\om(\alphaup,d(\betaup))=0
\end{equation}
 for $\alphaup,\betaup\in\{x,y,z\}$. If $\alphaup=\betaup$, then Equation (\ref{E3-2}) holds from the fact that $\om$ is anti-symmetric. So we may assume that $\alphaup\neq \betaup$. 
 
 Subcase 1: $(\alphaup,\betaup)=(x,y)$ or $(y,x)$. Note that $\om(y,z)=0=\om(x,z)$, so
 $\om(d(x),y)+\om(x,d(y))=\om(a_{13}z,y)+\om(x,-a_{13}z)=
 a_{13}\om(z,y)-a_{13}\om(x,z)=0$.
 
  Subcase 2: $(\alphaup,\betaup)=(x,z)$ or $(z,x)$. In this case,  $\om(d(x),z)+\om(x,d(z))=\om(a_{13}z,z)+\om(x,a_{33}z)=a_{13}\om(z,z)+a_{33}\om(x,z)=0.$
  
 Subcase 3: $(\alphaup,\betaup)=(y,z)$ or $(z,y)$. Similarly, $\om(d(y),z)+\om(y,d(z))=\om(-a_{13}z,z)+\om(y,a_{33}z)=0.$

Therefore, Equation (\ref{E3-2}) follows, completing the proof of the second assertion.
\end{proof}

Similar arguments can be applied to the remaining 3-dimensional $\om$-Lie algebras, so we summarize the result in the following Table \ref{3d}, without detailed proofs. Note that in Table \ref{3d}, $\g_{1}$ is the unique 1-dimensional Lie algebra, 
$\g_{2}$ is defined as Equation (\ref{g2}), and $\mathfrak{sl}_{2}(\C)$ denotes the
3-dimensional simple Lie algebra.

\begin{center}
\footnotesize{
\begin{longtable}{c|c|c|c|c}

\caption[Derivations of 3-dimensional $\om$-Lie algebras $\g$]{{\rm Derivations of 3-dimensional $\om$-Lie algebras $\g$}} \label{3d} \\
 \hline 
\endfirsthead

\multicolumn{5}{c}%
{\footnotesize  \tablename\ \thetable{}-- continued from previous page} \\
\hline
\endhead

\hline \multicolumn{5}{r}{{Continued on next page}} \\ \hline
\endfoot

\hline
\endlastfoot

$\g$ & Elements in $\Der(\g)$ & $\Der_{\om}(\g)=\Der(\g)$ &  Properties of $\Der(\g)$ & Remarks  \\ \hline
 
 $L_{1}$ &  $\begin{pmatrix}
    0  &  0&a  \\
    0  &  0& -a\\
    0&0&b
\end{pmatrix},a,b\in\C$ & True &  Soluble (not nilpotent) & $(\simeq) ~~\g_{2}$  \\ \hline

$L_{2}$ &  $\begin{pmatrix}
    a  &  0&0  \\
    0  &  0& 0\\
    0&0&-a
\end{pmatrix},a\in\C$ &True &  Abelian & $(\simeq) ~~\g_{1}$  \\ \hline

$ \begin{array}{c }
      A_{\alphaup}    \\
       (\alphaup\in\C) 
\end{array}$& $\begin{pmatrix}
    0  &  0&0  \\
    a  &  0& 0\\
    a/2&a&0
\end{pmatrix},a\in\C$ &True &  Abelian & $(\simeq) ~~\g_{1}$  \\ \hline

$B$ & $\begin{pmatrix}
    0  &  0&0  \\
    0  &  0& 0\\
    0&a&0
\end{pmatrix},a\in\C$ &True &  Abelian & $(\simeq) ~~\g_{1}$  \\ \hline

$ \begin{array}{c}
      C_{\alphaup}    \\
       (\alphaup\in\C-\{1,0,-1\}) 
\end{array}$ & $\begin{pmatrix}
    0  &  0&0  \\
    0  &  a& 0\\
    0&0&-a
\end{pmatrix},a\in\C$ &True &  Abelian & $(\simeq) ~~\g_{1}$  \\ \hline

$C_{1}  $ & $\begin{pmatrix}
    0  &  0&0  \\
    0  &  a& c\\
    0&b&-a
\end{pmatrix},a,b,c\in\C$ &True & Simple  & $(\simeq) ~~\mathfrak{sl}_{2}(\C)$  \\ \hline
\end{longtable}
}\end{center}

\section{\scshape Automorphisms in Dimension 3}

\setcounter{equation}{0}
\renewcommand{\theequation}
{4.\arabic{equation}}
\setcounter{theorem}{0}
\renewcommand{\thetheorem}
{4.\arabic{theorem}}

In this section we calculate automorphisms of $3$-dimensional $\om$-Lie algebras. We only give the detailed proofs for the cases of 
$L_{1}$ and $A_{\alpha}$, omitting the proofs for the remaining cases because of similarity of arguments. 

We begin with the computations of $\Aut(L_{1})$. 

For any $\sigmaup\in\Aut(L_{1})$, we assume that 
$\sigmaup=(a_{ij})\in\gl_{3}(\C)$ with $\det(\sigmaup)\neq 0.$
Then $\sigmaup(x)=a_{11}x+a_{12}y+a_{13}z,\sigmaup(y)=a_{21}x+a_{22}y+a_{23}z$ and 
$\sigmaup(z)=a_{31}x+a_{32}y+a_{33}z$. Since $[x,z]=0$, we have 
$0=\sigmaup([x,z])=[\sigmaup(x),\sigmaup(z)]=[a_{11}x+a_{12}y+a_{13}z,a_{31}x+a_{32}y+a_{33}z]=
(a_{11}a_{32}-a_{12}a_{31})y+(a_{12}a_{33}-a_{13}a_{32})z$. Thus 
$$a_{11}a_{32}=a_{12}a_{31} \textrm{ and } a_{12}a_{33}=a_{13}a_{32}.$$
Since $[x,y]=y$, we expand the equation $\sigmaup(y)=[\sigmaup(x),\sigmaup(y)]$ and derive
$$a_{21}=0,a_{22}=a_{11}a_{22}-a_{12}a_{21} \textrm{ and  } a_{23}=a_{12}a_{23}-a_{13}a_{22}.$$
Similarly, since $[y,z]=z$, we have $[\sigmaup(y),\sigmaup(z)]=\sigmaup(z)$. 
Expanding this equation we get
$$a_{31}=0,a_{32}=a_{21}a_{32}-a_{22}a_{31} \textrm{ and } a_{33}=a_{22}a_{33}-a_{23}a_{32}.$$
A direct computation shows that 
\begin{equation}
\label{ }
\sigmaup=\begin{pmatrix}
  1    & 0&a_{13}   \\
    0  & 1& -a_{13}\\
    0&0&a_{33} 
\end{pmatrix}, 
\end{equation}
where $a_{33}\neq 0$. This completes the characterization of arbitrary element in $\Aut(L_{1})$.

To describe $\Aut_{\om}(L_{1})$, we suppose $\sigmaup=\begin{pmatrix}
  1    & 0&a   \\
    0  & 1& -a\\
    0&0&b
\end{pmatrix}\in \Aut(L_{1})$. Then $\sigmaup(x)=x+az,\sigmaup(y)=y-az$ and
$\sigmaup(z)=bz$. Recall that $\om(x,z)=0=\om(y,z)$ and $\om(x,y)=1$, we see
that for any $u,v\in L_{1}$,
$$\om(u,v)=\om(\sigmaup(u),\sigmaup(v))$$
which shows that $\Aut_{\om}(L_{1})=\Aut(L_{1})$. Moreover, we can prove

\begin{prop}\label{pro4.1}
\begin{enumerate}
  \item $\Aut(L_{1})$ is a matrix Lie group.
  \item $\Aut(L_{1})=\exp(\Der(L_{1}))$, where $\exp(-)$ denotes the matrix exponential.
  \item $\Aut(L_{1})$ is a connected Lie group, i.e., path-connected as a topological space.
  \item $\Aut(L_{1})$ is a soluble Lie group.
\end{enumerate}
 \end{prop}

\begin{proof} 
(1)  We have seen that  $\Aut(L_{1})$ is a matrix Lie group because it is closed in the general linear group
$\GL(3,\C)$, see Sagle-Walde \cite[Proposition 7.1]{SW1973}.
 
(2) Recall that the matrix exponential  exp$:\gl_{3}(\C)\ra \GL(3,\C)$ was given by 
$$X\mapsto \sum_{k=0}^{\infty}\frac{X^{k}}{k!},$$
for any $3\times 3$ complex matrix $X$. By Sagle-Walde \cite[Proposition 7.3 (a)]{SW1973}, we have $\exp(tX)\in \Aut(L_{1})$ for any $t\in \mathbb{R}$ and any derivation $X\in \Der(L_{1})$. 
Thus the exponential mapping $\exp:\Der(L_{1})\ra \Aut(L_{1})$ does make sense, i.e., the image $\exp(\Der(L_{1}))\subseteq \Aut(L_{1})$ as a subset. To show $\exp(\Der(L_{1}))=\Aut(L_{1})$, we need only to show that 
$$\exp(\Der(L_{1}))\supseteq\Aut(L_{1}),$$
which means that  for any $\sigmaup=\begin{pmatrix}
  1    & 0&a   \\
    0  & 1& -a\\
    0&0&b
\end{pmatrix}\in\Aut(L_{1})$, it is sufficient to show that there exists some derivation $d\in\Der(L_{1})$ such that $\exp(d)=\sigmaup.$

\textsc{Subcase 1}: $b\neq 1$. 
Since nonzero number $b\neq 1$, it follows from elementary analysis that $e^{x}=b$ has a nonzero solution in $\C$, i.e., there exists a complex number $0\neq b_{0}\in\C$ such that $e^{b_{0}}=b.$ We define 
$$a_{0}:=\frac{a\cdot b_{0}}{e^{b_{0}}-1}$$
and consider the derivation  $$d:=\begin{pmatrix}
  0    & 0&a_{0}   \\
    0  & 0& -a_{0}\\
    0&0&b_{0}
\end{pmatrix}.$$
We make the convention that $d^{0}:=I_{3}$, the identity matrix.
By induction on $k$, we can show  that $d^{k}=b_{0}^{k-1}\cdot d$ for all $k\in\N^{+}$. Thus
\begin{eqnarray*}
\exp(d) & = & I_{3}+d/1!+d^{2}/2!+d^{3}/3!+\cdots \\
 & = & I_{3}+d(1+b_{0}/2!+b_{0}^{2}/3!+\cdots) \\
  & = & I_{3}+\frac{d}{b_{0}}(b_{0}+b_{0}^{2}/2!+b_{0}^{3}/3!+\cdots) \\
   & = & I_{3}+\frac{d}{b_{0}}(-1+1+b_{0}+b_{0}^{2}/2!+b_{0}^{3}/3!+\cdots) \\
     & = & I_{3}+\frac{d}{b_{0}}(e^{b_{0}}-1) \\
     &=&\begin{pmatrix}
  1    & 0&0 \\
    0  & 1& 0\\
    0&0&1
\end{pmatrix}+\begin{pmatrix}
  0    & 0&a_{0}  \cdot\frac{e^{b_{0}}-1}{b_{0}} \\
    0  & 0& -a_{0}\cdot\frac{e^{b_{0}}-1}{b_{0}}\\
    0&0&e^{b_{0}}-1
\end{pmatrix}\\
&=&\begin{pmatrix}
  1    & 0&a   \\
    0  & 1& -a\\
    0&0&b
\end{pmatrix}=\sigmaup.
\end{eqnarray*}

\textsc{Subcase 2}: $b=1$.  In this case we define
$$d:=\begin{pmatrix}
  0    & 0&a   \\
    0  & 0& -a\\
    0&0&0
\end{pmatrix}.$$
Then $d$ is nilpotent since $d^{2}=0$. 
Thus
$$\exp(d)=\begin{pmatrix}
  1    & 0&a \\
    0  & 1& -a\\
    0&0&1
\end{pmatrix}=\sigmaup.$$ 

The proof of the second assertion is completed.

(3) To see that  $\Aut(L_{1})$ is connected, we first note that the Lie algebra of $\Aut(L_{1})$ is just $\Der(L_{1})$, see
 \cite[Proposition 7.3 (b)]{SW1973}. Since $\GL(3,\C)$ is a matrix Lie group with $\gl_{3}(\C)$ as its Lie algebra, 
   $\Aut(L_{1})$ is an analytic subgroup of $\GL(3,\C)$, according to \cite[Definition 3.12]{Hal2000}, the definition of analytic 
   subgroup. It follows from \cite[Proposition 3.13]{Hal2000} that $\Aut(L_{1})$ is path-connected, i.e., 
 $\Aut(L_{1})$ is a connected Lie group.
 
 (4) Since $\Aut(L_{1})$ is  connected and its Lie algebra is soluble, it follows from  Sagle-Walde 
 \cite[Theorem 10.9 (b)]{SW1973} that $\Aut(L_{1})$ is soluble. 
\end{proof}

\begin{rem} {\rm
We remark that there also exists an approach to compute the exponential of any matrix, using the SN-Decomposition, see Hall \cite[Sections 2.2.1-2.2.3]{Hal2000} for the details.
}\end{rem}

Similar arguments can be applied to the cases of $L_{2},B,$ and $C_{\alphaup}$. We will see below that 
the computation of $\Aut(A_{\alphaup})$ is more complicated. Here a ``trick'' coming from Linear Algebra will be useful. 

Suppose $\sigmaup=(a_{ij})_{3\times 3}\in\Aut(A_{\alphaup})$
 and $\sigmaup(x)=a_{11}x+a_{12}y+a_{13}z,\sigmaup(y)=a_{21}x+a_{22}y+a_{23}z$ and 
$\sigmaup(z)=a_{31}x+a_{32}y+a_{33}z$.
Since $[x,y]=x$, $\sigmaup(x)=\sigmaup([x,y])=[\sigmaup(x),\sigmaup(y)]$. Note that the relations $[y,z]=z+\alphaup x$ and $[x,z]=x+y$. It follows that 
\begin{eqnarray*}
a_{11}x+a_{12}y+a_{13}z & = & [a_{11}x+a_{12}y+a_{13}z,a_{21}x+a_{22}y+a_{23}z] \\
 & = & (a_{11}a_{22}+a_{11}a_{23}-a_{12}a_{21}+\alphaup a_{12}a_{23}-a_{13}a_{21}-\alphaup a_{13}a_{22})x+\\
 &&(a_{11}a_{23}-a_{13}a_{21})y+(a_{12}a_{23}-a_{13}a_{22})z.
\end{eqnarray*}
The algebraic independence of $x,y,z$ implies that 
\begin{eqnarray}
a_{11} & = & a_{11}a_{22}+a_{11}a_{23}-a_{12}a_{21}+\alphaup a_{12}a_{23}-a_{13}a_{21}-\alphaup a_{13}a_{22}\\
&=&A_{33}+A_{32}+\alphaup A_{31} \nonumber\\
a_{12} & = & a_{11}a_{23}-a_{13}a_{21}=A_{32}\\
a_{13}&=&a_{12}a_{23}-a_{13}a_{22}=A_{31}.
\end{eqnarray}
where $A_{ij}$ denotes the minor determinant obtained by deleting 
the $i$-th row and $j$-th column in $\det(\sigmaup)$.

Similarly,  it follows from $[\sigmaup(x),\sigmaup(z)]=\sigmaup(x+y)$ that 
\begin{eqnarray}
a_{11}+a_{21} & = & a_{11}a_{32}+a_{11}a_{33}-a_{12}a_{31}+\alphaup a_{12}a_{33}-a_{13}a_{31}-\alphaup a_{13}a_{32} \\
&=&A_{23}+A_{22}+\alphaup A_{21}\nonumber\\
a_{12}+a_{22} & = & a_{11}a_{33}-a_{13}a_{31}=A_{22}\\
a_{13}+a_{23}&=&a_{12}a_{33}-a_{13}a_{32}=A_{21}.
\end{eqnarray}
We expand $\sigmaup([y,z])=\sigmaup(z+\alphaup x)$ to get
\begin{eqnarray}
a_{31}+\alphaup a_{11} & = & a_{21}a_{32}+a_{21}a_{33}-a_{22}a_{31}+\alphaup a_{22}a_{33}-a_{23}a_{31}-\alphaup a_{23}a_{32} \\
&=& A_{13}+A_{12}+\alphaup A_{11}\nonumber\\
a_{32}+\alphaup a_{12} & = & a_{21}a_{33}-a_{23}a_{31}=A_{12}\\
a_{33}+\alphaup a_{13}&=& a_{22}a_{33}-a_{23}a_{32}=A_{11}.
\end{eqnarray}

We denote by $\sigmaup^{*}$  the matrix for which  $\sigmaup\cdot \sigmaup^{*}=\det(\sigmaup)\cdot I_{3}=\sigmaup^{*}\cdot \sigmaup$. Then
\begin{eqnarray*}
\begin{pmatrix}
    \det(\sigmaup)  &    0& 0\\
   0   &  \det(\sigmaup)&0\\
   0&0&\det(\sigmaup)
\end{pmatrix} & = & \begin{pmatrix}
     a_{11} & a_{12} & a_{13}    \\
           a_{21} & a_{22} & a_{23}    \\
     a_{31} & a_{32} & a_{33}      
\end{pmatrix}\cdot \begin{pmatrix}
     A_{11} & -A_{21} & A_{31}    \\
           -A_{12} & A_{22} & -A_{32}    \\
     A_{13} & -A_{23} & A_{33}      
\end{pmatrix} \\
 & = & \begin{pmatrix}
     a_{11} & a_{12} & a_{13}    \\
           a_{21} & a_{22} & a_{23}    \\
     a_{31} & a_{32} & a_{33}      
\end{pmatrix}\cdot \begin{pmatrix}
     a_{33}+\alpha a_{13} & -(a_{13}+a_{23}) & a_{13}    \\
           -(a_{32}+\alpha a_{12}) & a_{12}+a_{22} & -a_{12}    \\
      A_{13}& -A_{23} & a_{11}-a_{12}-\alpha a_{13}      
\end{pmatrix} 
\end{eqnarray*}
where $A_{13}=a_{31}+\alpha a_{11}-a_{32}-\alpha a_{12}-\alpha(a_{33}+\alpha a_{13})$
and $A_{23}=a_{11}+a_{21}-a_{12}-a_{22}-\alpha a_{13}-\alpha a_{23}$.
This equality gives rise to 9 equations involving 9 unknowns: $\{a_{ij} \mid 1\leqslant i,j\leqslant 3\}$. It follows from these 9 equations that 
\begin{eqnarray}
a_{11}a_{33} & = & \det(\sigmaup) \\
a_{21}a_{33}-a_{22}a_{32} & = & 0\\
a_{22}^{2}&=& \det(\sigmaup)\\
2a_{31}a_{33}-a_{32}^{2}+\alphaup a_{11}a_{33}-a_{32}a_{33}-\alphaup a_{33}^{2}&=&0\\
a_{22}a_{32}-a_{21}a_{33}+a_{22}a_{33}&=&\det(\sigmaup)\\
a_{13}=a_{12}=a_{23}&=&0.
\end{eqnarray}
Note that $\det(\sigmaup)\neq 0$, so $a_{11},a_{22},a_{33}$ are not zero.
A direct calculation shows that 
$a_{11}=a_{22}=a_{33}=1, a_{21}=a_{32},$ and $2a_{31}-a_{32}^{2}-a_{32}=0.$
Therefore,
$$\sigmaup=\begin{pmatrix}
     1 &0&0    \\
     a_{21} &1&0\\
     (a_{21}^{2}+a_{21})/2&a_{21}&1  
\end{pmatrix}$$
which is exactly what we want. Furthermore, from the generating relations of $A_{\alphaup}$ we observe that 
$$\Aut_{\om}(A_{\alphaup})=\Aut(A_{\alphaup}).$$ With a similar argument as Proposition \ref{pro4.1} we obtain

\begin{prop}\label{pro4.3}
\begin{enumerate}
  \item $\Aut(A_{\alphaup})$ is a connected, abelian matrix Lie group.
  \item $\Aut(A_{\alphaup})=\exp(\Der(A_{\alphaup}))$.
\end{enumerate}
 \end{prop}

\begin{proof}
(1) The proof is similar to Proposition \ref{pro4.1}. It is immediate to see that 
$\Aut(A_{\alphaup})$ is a connected matrix Lie group.  For any elements $\sigmaup,\sigmaup'\in\Aut(A_{\alphaup})$
we see that  $\sigmaup\cdot\sigmaup'= \sigmaup'\cdot\sigmaup.$ Thus
$\Aut(A_{\alphaup})$ is abelian. 

(2) Since $\Aut(A_{\alphaup})\supseteq\exp(\Der(A_{\alphaup}))$, we need only to show that 
$\Aut(A_{\alphaup})\subseteq\exp(\Der(A_{\alphaup}))$. For any derivation
$$d=\begin{pmatrix}
    0  &  0&0  \\
    a  &  0& 0\\
    a/2&a&0
\end{pmatrix}\in\Der(A_{\alphaup}),$$
it follows that $$d^{2}=\begin{pmatrix}
    0  &  0&0  \\
    0  &  0& 0\\
    a^{2}&0&0
\end{pmatrix}$$
and $d^{3}=0$, so $d$ is nilpotent. Thus 
$$\exp(d)=I_{3}+d+\frac{d^{2}}{2}=\begin{pmatrix}
    1  &  0&0  \\
    a  &  1& 0\\
    (a^{2}+a)/2&a&1
\end{pmatrix}$$
which means that we can find a derivation $d\in \Der(A_{\alphaup})$ such that 
$\exp(d)=\sigmaup$ for any automorphism $\sigmaup\in\Aut(A_{\alphaup}).$
 \end{proof}

We summarize a characterization on automorphisms of 3-dimensional $\om$-Lie algebras as following Table \ref{3a}.

\begin{center}
\footnotesize{
\begin{longtable}{c|c|c|c|c}

\caption[Automorphisms of 3-dimensional $\om$-Lie algebras $\g$]{{\rm Automorphisms of 3-dimensional $\om$-Lie algebras $\g$}} \label{3a} \\
 \hline 
\endfirsthead

\multicolumn{5}{c}%
{\footnotesize  \tablename\ \thetable{}-- continued from previous page} \\
\hline
\endhead

\hline \multicolumn{5}{r}{{Continued on next page}} \\ \hline
\endfoot

\hline
\endlastfoot    

 $\g$ & Elements in $\Aut(\g)$ & $\Aut_{\om}(\g)=\Aut(\g)$ &  Properties of $\Aut(\g)$ & Unipotent?  \\ \hline
 
 $L_{1}$ &  $\begin{pmatrix}
    1  &  0&a  \\
    0  &  1& -a\\
    0&0&b
\end{pmatrix},b\neq 0,a\in\C$ & True &  Soluble & False  \\ \hline

$L_{2}$ &  $\begin{pmatrix}
    a  &  0&0  \\
    0  &  1& 0\\
    0&0&1/a
\end{pmatrix},0\neq a\in\C$ &True &  Abelian & False\\ \hline

$ \begin{array}{c }
      A_{\alphaup}    \\
       (\alphaup\in\C) 
\end{array}$& $\begin{pmatrix}
    1  &  0&0  \\
    a  &  1& 0\\
    (a^{2}+a)/2&a&1
\end{pmatrix},a\in\C$ &True &  Abelian & True  \\ \hline

$B$ & $\begin{pmatrix}
    1 &  0&0  \\
    0  &  a& 0\\
    0&b&a
\end{pmatrix},a^{2}=1,b\in\C$ &True &  Abelian & True \\ \hline

$ \begin{array}{c}
      C_{\alphaup}    \\
       (\alphaup\in\C-\{1,0,-1\}) 
\end{array}$ & $\begin{pmatrix}
    1  &  0&0  \\
    0  &  a& 0\\
    0&0&1/a
\end{pmatrix},0\neq a\in\C$ &True &  Abelian & False  \\ \hline

$C_{1}  $ & $ \begin{array}{c}
      \begin{pmatrix}
    1  &  0&0  \\
    0  &  a& c\\
    0&d&b
\end{pmatrix}, \\
ab-cd=1
\end{array}$ &True & $(\simeq) ~~\SL_{2}(\C)$  & False   \\ \hline

\end{longtable}
}\end{center}

By Propositions \ref{pro4.1} (2) and \ref{pro4.3} (2), one might ask whether $\exp(\Der(\g))=\Aut(\g)$ always holds for all 3-dimensional $\om$-Lie algebra $\g$? Unfortunately, the answer to this question is ``no'' in general, as the following example shows.

\begin{exam}{\rm
Consider $\exp(\Der(B))$ and $\Aut(B)$. For any $d\in \Der(B)$, we may take 
$$d=\begin{pmatrix}
      0&0&0    \\
      0&0&0\\
      0&a&0  
\end{pmatrix}.$$
Then $d^{2}=0$ is nilpotent. Thus $\exp(d)=d^{0}+d^{1}/1!=I_{3}+d=\begin{pmatrix}
      1&0&0    \\
      0&1&0\\
      0&a&1
\end{pmatrix}.$ On the other hand, 
$$\sigmaup:=\begin{pmatrix}
      1&0&0    \\
      0&-1&0\\
      0&a&-1 
\end{pmatrix}\in \Aut(B)$$
but $\sigmaup$ doesn't belong to $\exp(\Der(B))$. Hence $\exp(\Der(\g))\subset\Aut(\g)$ but $\exp(\Der(\g))\neq\Aut(\g)$.
}
 \end{exam}

\section{\scshape Derivations and Automorphisms in Dimension 4}

\setcounter{equation}{0}
\renewcommand{\theequation}
{5.\arabic{equation}}
\setcounter{theorem}{0}
\renewcommand{\thetheorem}
{5.\arabic{theorem}}

Suppose $\g$ is a 4-dimensional nontrivial $\om$-Lie algebra over $\C$ and 
$\{x,y,z,e\}$ is a basis for $\g.$ The characterizations of $\Der(\g)$ and $\Aut(\g)$ are summarized in the following 
Tables \ref{4d} and \ref{4a} respectively. Note that all parameters in Tables \ref{4d} and \ref{4a} belong to $\C$.

\begin{center}
\footnotesize{
\begin{longtable}{c|c|c|c}

\caption[Derivations of 4-dimensional $\om$-Lie algebras $\g$]{{\rm Derivations of 4-dimensional $\om$-Lie algebras $\g$}} \label{4d} \\
 \hline 
\endfirsthead

\multicolumn{4}{c}
{\footnotesize  \tablename\ \thetable{}-- continued from previous page} \\
\hline
\endhead

\hline \multicolumn{4}{r}{{Continued on next page}} \\ \hline
\endfoot

\hline
\endlastfoot    
    
 $\g$ & Relations in $\g$ & Elements in $\Der(\g)$&   $\dim(\Der(\g))$  \\ \hline

$L_{1,1}$&$ \begin{array}{c}
      [x,y]=y, [y,z]=z,  \\
      ~[e,y]=-y,\\
      \om(x,y)=1\\
\end{array}$& $\begin{pmatrix}
    0  &  0&a&b \\
    0 &  0& -a&-b\\
    0&0&c& d\\
    0&0&h&f
\end{pmatrix}$ & 6 \\ \hline

$L_{1,2}$&$ \begin{array}{c}
      [x,y]=y, [y,z]=z,  \\
      ~[e,x]=z,[e,y]=-e,\\
      \om(x,y)=1\\
\end{array}$& $\begin{pmatrix}
    0  &  0&a&b \\
    0 &  0& b-a&-b\\
    0&0&c& 0\\
    0&0&d&c
\end{pmatrix}$ & 4  \\ \hline

$L_{1,3}$&$ \begin{array}{c}
      [x,y]=y, [y,z]=z,  \\
      ~[e,x]=y,[e,y]=-e,\\
      \om(x,y)=1,\om(e,x)=1\\
\end{array}$& $\begin{pmatrix}
    0  &  0&0&0 \\
    0 &  0& 0&0\\
    0&0&a& 0\\
    0&0&b&0
\end{pmatrix}$ & 2  \\ \hline

$L_{1,4}$&$ \begin{array}{c}
      [x,y]=y, [y,z]=z,  \\
      ~[e,x]=y+z,[e,y]=-e,\\
      \om(x,y)=1,\om(e,x)=1\\
\end{array}$& $\begin{pmatrix}
    0  &  0&a&0 \\
    0 &  0& -a&0\\
    0&0&a& 0\\
    0&0&b&0
\end{pmatrix}$ & 2  \\ \hline

$L_{1,5}$&$ \begin{array}{c}
      [x,y]=y, [y,z]=z,  \\
      ~[e,x]=e,[e,y]=-e,\\
      \om(x,y)=1\\
\end{array}$& $\begin{pmatrix}
    0  &  0&a&-2d \\
    0 &  0& -a&d\\
    0&0&b& 0\\
    0&0&0&c
\end{pmatrix}$ & 4  \\ \hline

$L_{1,6}$&$ \begin{array}{c}
      [x,y]=y, [y,z]=z,  \\
      ~[e,x]=e+y,[e,y]=-e,\\
      \om(x,y)=1=\om(e,x)\\
\end{array}$& $\begin{pmatrix}
    -a  &  0&c&-2a \\
    0 &  0& -c&a\\
    0&0&b& 0\\
    0&0&c&-a
\end{pmatrix}$ & 3  \\ \hline

$L_{1,7}$&$ \begin{array}{c}
      [x,y]=y, [y,z]=z,  \\
      ~[e,x]=e,[e,y]=z-e,\\
      \om(x,y)=1\\
\end{array}$& $\begin{pmatrix}
    0  &  0&c&b \\
    0 &  0& b-c&-b/2\\
    0&0&a& 0\\
    0&0&0&a
\end{pmatrix}$ & 3  \\ \hline

$L_{1,8}$&$ \begin{array}{c}
      [x,y]=y, [y,z]=z,  \\
      ~[e,x]=e+y,[e,y]=z-e,\\
      \om(x,y)=1=\om(e,x)\\
\end{array}$& $\begin{pmatrix}
    a  &  0&b&-2a \\
    0 &  0& -b-2a&a\\
    0&0&a& 0\\
    0&0&b+2a&-a
\end{pmatrix}$ & 2  \\ \hline

$L_{2,1}$&$ \begin{array}{c}
      [x,z]=y, [y,z]=z,  \\
      ~[e,y]=-e,\\
      \om(x,z)=1\\
\end{array}$& $\begin{pmatrix}
    a  &  0&0&0 \\
    0 &  0& 0&0\\
    0&0&-a& b\\
    0&0&0&c
\end{pmatrix}$ & 3  \\ \hline

$L_{2,2}$&$ \begin{array}{c}
      [x,z]=y, [y,z]=z,  \\
      ~[e,y]=-e,[e,x]=z,\\
      \om(x,z)=1\\
\end{array}$& $\begin{pmatrix}
    a  &  0&0&0 \\
    0 &  0& 0&0\\
    0&0&-a& 0\\
    0&0&0&-2a
\end{pmatrix}$ & 1  \\ \hline

$L_{2,3}$&$ \begin{array}{c}
      [x,z]=y, [y,z]=z,  \\
      ~[e,y]=-e,[e,x]=e,\\
      \om(x,z)=1\\
\end{array}$& $\begin{pmatrix}
    0 &  0&0&a \\
    0 &  0& 0&-a\\
    0&0&0& a\\
    0&0&0&b
\end{pmatrix}$ & 2  \\ \hline

$L_{2,4}$&$ \begin{array}{c}
      [x,z]=y, [y,z]=z,  \\
      ~[e,y]=-e,[e,x]=e+z,\\
      \om(x,z)=1\\
\end{array}$& $\begin{pmatrix}
    2a &  a&a&a \\
    0 &  0& 0&-a\\
    0&0&-2a& a\\
    0&0&0&4a
\end{pmatrix}$ & 1  \\ \hline

$\widetilde{B}$&$ \begin{array}{c}
      [x,y]=y, [x,z]=y+z,[y,z]=x,  \\
      ~[e,y]=-e,[e,x]=-2e,\\
      \om(y,z)=2\\
\end{array}$& $\begin{pmatrix}
    0 &  0&0&0 \\
    0 &  0& 0&0\\
    0&a&0& 0\\
    0&0&0&b
\end{pmatrix}$ & 2  \\ \hline

$\begin{array}{c}
       E_{1,\alphaup}   \\
       (\alphaup \neq 0,1)
\end{array}$ & $ \begin{array}{c}
      [x,y]=y, [y,z]=z,  \\
      ~[e,y]=-e,[e,x]=\alphaup e,\\
      \om(x,y)=1\\
\end{array}$& $\begin{pmatrix}
    0 &  0&a&-(\alphaup+1)b \\
    0 &  0& -\alphaup&b\\
    0&0&c& 0\\
    0&0&0&d
\end{pmatrix}$ & 4  \\ \hline

$\begin{array}{c}
       F_{1,\alphaup}   \\
       (\alphaup \neq 0,1)
\end{array}$ &$ \begin{array}{c}
      [x,y]=y, [y,z]=z,  \\
      ~[e,y]=-e,[e,x]=\alphaup e+y,\\
      \om(x,y)=1=\om(e,x)\\
\end{array}$& $\begin{pmatrix}
    0 &  0&\alphaup a&0 \\
    0 &  0& -\alphaup a&0\\
    0&0&b& 0\\
    0&0&a&0
\end{pmatrix}$ & 2  \\ \hline

$\begin{array}{c}
       G_{1,\alphaup}       \end{array}$ &$ \begin{array}{c}
      [x,y]=y, [y,z]=z,  \\
      ~[e,y]=x-e,[e,x]=e+\alphaup y,\\
      \om(x,y)=1,\om(e,x)=\alphaup\\
\end{array}$& $\begin{pmatrix}
    -(\alphaup a)/2&  0&0&a \\
    a&  0& 0&-a/2\\
    0&0&b& 0\\
    -\alphaup a&0&0&(\alphaup a)/2
\end{pmatrix}$ & 2  \\ \hline

$\begin{array}{c}
       H_{1,\alphaup}       \end{array}$ &$ \begin{array}{c}
      [x,y]=y, [y,z]=z,  \\
      ~[e,y]=x+z-e,[e,x]=e+\alphaup y,\\
      \om(x,y)=1,\om(e,x)=\alphaup\\
\end{array}$& $\begin{pmatrix}
    -(\alpha a)/2&  0&a-b&a \\
    a&  0& b&-a/2\\
    0&0&(\alphaup-2) a/2+b& 0\\
    -\alphaup a&0&-\alphaup b&(\alphaup a)/2
\end{pmatrix}$ & 2  \\ \hline

$\begin{array}{c}
       \widetilde{A}_{\alphaup}       \end{array}$ &$ \begin{array}{c}
      [x,y]=x, [x,z]=x+y,\\
      ~[y,z]=z+\alphaup x,  \\
      ~[e,z]=e, \om(y,z)=-1\\
\end{array}$& $\begin{pmatrix}
    0&  0&0&0 \\
    a&  0& 0&0\\
    a/2&a&0& 0\\
    0&0&0&b
\end{pmatrix}$ & 2  \\ \hline

$\begin{array}{c}
       \widetilde{C}_{\alphaup}      \\
       (\alphaup\neq 0,-1,1)
        \end{array}$ &$ \begin{array}{c}
      [x,y]=y, [x,z]=\alphaup z,\\
      ~[y,z]=x, [e,x]=-(1+\alphaup)e,\\
       \om(y,z)=1+\alphaup\\
\end{array}$& $\begin{pmatrix}
    0&  0&0&0 \\
    0&  a& 0&0\\
    0&0&-a& 0\\
    0&0&0&b
\end{pmatrix}$ & 2  \\ \hline

$\begin{array}{c}
       \widetilde{C}_{1}        \end{array}$ &$ \begin{array}{c}
      [x,y]=y, [x,z]=z,\\
      ~[y,z]=x, [e,x]=-2e,\\
       \om(y,z)=2\\
\end{array}$& $\begin{pmatrix}
    0&  0&0&0 \\
    0&  a& c&0\\
    0&d&-a& 0\\
    0&0&0&b
\end{pmatrix}$ & 4  \\ \hline

\end{longtable}
}\end{center}

As in the case of 3-dimensional $\om$-Lie algebras, direct computations show that

\begin{prop}\label{pro5.1}
$\Der_{\om}(\g)=\Der(\g)$ for any 4-dimensional nontrivial $\om$-Lie algebra $\g$, except for $\g=L_{1,6}$ and $L_{1,8}$.
 \end{prop}

\begin{proof}
Here we only prove that $\Der_{\om}(L_{1,6})\neq\Der(L_{1,6})$. Consider a derivation
$$D:=\begin{pmatrix}
    -a  &  0&c&-2a \\
    0 &  0& -c&a\\
    0&0&b& 0\\
    0&0&c&-a
\end{pmatrix}\in \Der(L_{1,6})$$
with $a\neq 0.$ Then
\begin{eqnarray*}
 &&\om(D(x),y)+\om(x,D(y))\\
 & = & \om(-ax+cz-2ae,y)+\om(x,-cz+ae)\\
 &=&-a+(-a)=-2a\neq 0.
\end{eqnarray*}
Thus $D$ is not an $\om$-derivation. 

Similar arguments show that $\Der_{\om}(L_{1,8})\neq\Der(L_{1,8})$.
\end{proof}

\begin{center}
\footnotesize{
\begin{longtable}{c|c|c}

\caption[Automorphisms of 4-dimensional $\om$-Lie algebras $\g$]{{\rm Automorphisms of 4-dimensional $\om$-Lie algebras $\g$}} \label{4a} \\
 \hline 
\endfirsthead

\multicolumn{3}{c}
{\footnotesize  \tablename\ \thetable{}-- continued from previous page} \\
\hline
\endhead

\hline \multicolumn{3}{r}{{Continued on next page}} \\ \hline
\endfoot

\hline
\endlastfoot    
    
 $\g$ & Relations in $\g$ & Elements in $\Aut(\g)$  \\ \hline

$L_{1,1}$&$ \begin{array}{c}
      [x,y]=y, [y,z]=z,  \\
      ~[e,y]=-y,\\
      \om(x,y)=1\\
\end{array}$& $\begin{pmatrix}
    1  &  0&a&b \\
    0 &  1& -a&-b\\
    0&0&c& d\\
    0&0&h&f
\end{pmatrix},~~dh-cf\neq0$   \\ \hline

$L_{1,2}$&$ \begin{array}{c}
      [x,y]=y, [y,z]=z,  \\
      ~[e,x]=z,[e,y]=-e,\\
      \om(x,y)=1\\
\end{array}$& $\begin{pmatrix}
    1  &  0&a&b \\
    0 &  1& b-a&-b\\
    0&0&c& 0\\
    0&0&d&c
\end{pmatrix}$,~~ $c\neq0$  \\ \hline

$L_{1,3}$&$ \begin{array}{c}
      [x,y]=y, [y,z]=z,  \\
      ~[e,x]=y,[e,y]=-e,\\
      \om(x,y)=1,\om(e,x)=1\\
\end{array}$& $\begin{pmatrix}
    1  &  0&0&0 \\
    0 &  1& 0&0\\
    0&0&a& 0\\
    0&0&b&1
\end{pmatrix},~~a\neq 0$  \\ \hline

$L_{1,4}$&$ \begin{array}{c}
      [x,y]=y, [y,z]=z,  \\
      ~[e,x]=y+z,[e,y]=-e,\\
      \om(x,y)=1,\om(e,x)=1\\
\end{array}$& $\begin{pmatrix}
    1 &  0&a&0 \\
    0 &  1& -a&0\\
    0&0&a+1& 0\\
    0&0&b&1
\end{pmatrix},~~a\neq -1$   \\ \hline

$L_{1,5}$&$ \begin{array}{c}
      [x,y]=y, [y,z]=z,  \\
      ~[e,x]=e,[e,y]=-e,\\
      \om(x,y)=1\\
\end{array}$& $\begin{pmatrix}
    1  &  0&a&-2d \\
    0 &  1& -a&d\\
    0&0&b& 0\\
    0&0&0&c
\end{pmatrix},~~bc\neq 0$   \\ \hline

$L_{1,6}$&$ \begin{array}{c}
      [x,y]=y, [y,z]=z,  \\
      ~[e,x]=e+y,[e,y]=-e,\\
      \om(x,y)=1=\om(e,x)\\
\end{array}$& $\begin{pmatrix}
    b  &  0&a&1/b-b \\
    0 &  1& -a&1-1/b\\
    0&0&c& 0\\
    0&0&a&1/b
\end{pmatrix},b\neq0\neq c$   \\ \hline

$L_{1,7}$&$ \begin{array}{c}
      [x,y]=y, [y,z]=z,  \\
      ~[e,x]=e,[e,y]=z-e,\\
      \om(x,y)=1\\
\end{array}$& $\begin{pmatrix}
    1  &  0&c&-2b \\
    0 &  1& -2b-c&b\\
    0&0&a& 0\\
    0&0&0&a
\end{pmatrix},~~a\neq 0$   \\ \hline

$L_{1,8}$&$ \begin{array}{c}
      [x,y]=y, [y,z]=z,  \\
      ~[e,x]=e+y,[e,y]=z-e,\\
      \om(x,y)=1=\om(e,x)\\
\end{array}$& $\begin{pmatrix}
    1/a  &  0&a+b-1/a&a-1/a \\
    0 &  1& -b&1-a\\
    0&0&a& 0\\
    0&0&b&a
\end{pmatrix},~~a\neq0$   \\ \hline

$L_{2,1}$&$ \begin{array}{c}
      [x,z]=y, [y,z]=z,  \\
      ~[e,y]=-e,\\
      \om(x,z)=1\\
\end{array}$& $\begin{pmatrix}
    a  &  0&0&0 \\
    0 &  1& 0&0\\
    0&0&1/a& b\\
    0&0&0&c
\end{pmatrix},~~a\neq0\neq c$   \\ \hline

$L_{2,2}$&$ \begin{array}{c}
      [x,z]=y, [y,z]=z,  \\
      ~[e,y]=-e,[e,x]=z,\\
      \om(x,z)=1\\
\end{array}$& $\begin{pmatrix}
    a  &  0&0&0 \\
    0 &  1& 0&0\\
    0&0&1/a& 0\\
    0&0&0&1/a^{2}
\end{pmatrix},~~a\neq0$   \\ \hline

$L_{2,3}$&$ \begin{array}{c}
      [x,z]=y, [y,z]=z,  \\
      ~[e,y]=-e,[e,x]=e,\\
      \om(x,z)=1\\
\end{array}$& $\begin{pmatrix}
    1&  0&0&a \\
    0 &  1& 0&-a\\
    0&0&1& a\\
    0&0&0&b
\end{pmatrix},~~b\neq0$   \\ \hline

$L_{2,4}$&$ \begin{array}{c}
      [x,z]=y, [y,z]=z,  \\
      ~[e,y]=-e,[e,x]=e+z,\\
      \om(x,z)=1\\
\end{array}$& $\begin{pmatrix}
    a &  \frac{a-1}{2}&\frac{a^{2}-1}{4a}&\frac{(a-1)(a+1)^{2}}{8a^{2}} \\
    0 &  1& 0&\frac{1-a^{2}}{4a^{2}}\\
    0&0&a^{-1}& \frac{a-1}{2a^{2}}\\
    0&0&0&a^{-2}
\end{pmatrix}$, $a\neq 0$   \\ \hline

$\widetilde{B}$&$ \begin{array}{c}
      [x,y]=y, [x,z]=y+z,[y,z]=x,  \\
      ~[e,y]=-e,[e,x]=-2e,\\
      \om(y,z)=2\\
\end{array}$& $\begin{pmatrix}
    1 &  0&0&0 \\
    0 &  a& 0&0\\
    0&c&a& 0\\
    0&0&0&b
\end{pmatrix},a^{2}=1,b\neq 0$   \\ \hline

$\begin{array}{c}
       E_{1,\alphaup}   \\
       (\alphaup \neq 0,1)
\end{array}$ & $ \begin{array}{c}
      [x,y]=y, [y,z]=z,  \\
      ~[e,y]=-e,[e,x]=\alphaup e,\\
      \om(x,y)=1\\
\end{array}$& $\begin{pmatrix}
    1 &  0&a&-(\alphaup+1)b \\
    0 &  1& -a&b\\
    0&0&c& 0\\
    0&0&0&d
\end{pmatrix},~~cd\neq 0$   \\ \hline

$\begin{array}{c}
       F_{1,\alphaup}   \\
       (\alpha \neq 0,1)
\end{array}$ &$ \begin{array}{c}
      [x,y]=y, [y,z]=z,  \\
      ~[e,y]=-e,[e,x]=\alphaup e+y,\\
      \om(x,y)=1=\om(e,x)\\
\end{array}$& $\begin{pmatrix}
    1 &  0& a&0 \\
    0 &  1& -a&0\\
    0&0&b& 0\\
    0&0&a/\alphaup&1
\end{pmatrix},b\neq 0$   \\ \hline

$\begin{array}{c}
       G_{1,0}       \end{array}$ &$ \begin{array}{c}
      [x,y]=y, [y,z]=z,  \\
      ~[e,y]=x-e,[e,x]=e,\\
      \om(x,y)=1\\
\end{array}$& $\begin{pmatrix}
    1&  0&0&b \\
    b&  1& 0&(b^{2}-b)/2\\
    0&0&a& 0\\
    0&0&0&1
\end{pmatrix},a\neq 0$   \\ \hline

$\begin{array}{c}
       G_{1,\alphaup}  \\
       (\alphaup\neq0)     \end{array}$ &$ \begin{array}{c}
      [x,y]=y, [y,z]=z,  \\
      ~[e,y]=x-e,[e,x]=e+\alphaup y,\\
      \om(x,y)=1,\om(e,x)=\alphaup\\
\end{array}$& $\begin{array}{c}\begin{pmatrix}
    a'-\alpha a&  0&0&a \\
    a&  1& 0&\frac{1-a'}{\alphaup}\\
    0&0&b& 0\\
    -\alphaup a&0&0&a'
\end{pmatrix},\\
\textrm{ where } a'=\frac{\alphaup a\pm \sqrt{\alphaup^{2}a^{2}-4\alphaup a^{2}+4}}{2}, b\neq 0
\end{array}$   \\ \hline

$\begin{array}{c}
       H_{1,0}       \end{array}$ &$ \begin{array}{c}
      [x,y]=y, [y,z]=z,  \\
      ~[e,y]=x+z-e,[e,x]=e,\\
      \om(x,y)=1\\
\end{array}$& $\begin{pmatrix}
    1&  0&1-b&a \\
    a& 1& a+b-1&(a^{2}-a)/2\\
    0&0&b& 0\\
    0&0&0&1
\end{pmatrix},~~b\neq 0$   \\ \hline

$\begin{array}{c}
       H_{1,\alphaup}  \\
      (\alphaup\neq0)     \end{array}$ &$ \begin{array}{c}
      [x,y]=y, [y,z]=z,  \\
      ~[e,y]=x+z-e,[e,x]=e+\alphaup y,\\
      \om(x,y)=1,\om(e,x)=\alphaup\\
\end{array}$& $\begin{array}{c}
      \begin{pmatrix}
    a'-\alphaup a&  0&b&a \\
    a&  1& a-b&(1-a')/\alphaup\\
    0&0&a'-b& 0\\
    -\alphaup a&0&\alphaup (b-a)&a'
\end{pmatrix},\\
\textrm{where } a'=(\alphaup a\pm \sqrt{\alphaup^{2}a^{2}-4\alphaup a^{2}+4})/2\\
a'-b\neq 0
\end{array}$   \\ \hline

$\begin{array}{c}
       \widetilde{A}_{\alphaup}       \end{array}$ &$ \begin{array}{c}
      [x,y]=x, [x,z]=x+y,\\
      ~[y,z]=z+\alphaup x,  \\
      ~[e,z]=e, \om(y,z)=-1\\
\end{array}$& $\begin{pmatrix}
    1&  0&0&0 \\
    a&  1& 0&0\\
    (a^{2}+a)/2&a&1& 0\\
    0&0&0&b
\end{pmatrix},~~b\neq 0$   \\ \hline

$\begin{array}{c}
       \widetilde{C}_{\alphaup}      \\
       (\alphaup\neq 0,-1,1)
        \end{array}$ &$ \begin{array}{c}
      [x,y]=y, [x,z]=\alphaup z,\\
      ~[y,z]=x, [e,x]=-(1+\alphaup)e,\\
       \om(y,z)=1+\alphaup\\
\end{array}$& $\begin{pmatrix}
    1&  0&0&0 \\
    0&  a& 0&0\\
    0&0&1/a& 0\\
    0&0&0&b
\end{pmatrix},~~a\neq0\neq b$   \\ \hline

$\begin{array}{c}
       \widetilde{C}_{1}        \end{array}$ &$ \begin{array}{c}
      [x,y]=y, [x,z]=z,\\
      ~[y,z]=x, [e,x]=-2e,\\
       \om(y,z)=2\\
\end{array}$& $\begin{pmatrix}
    1&  0&0&0 \\
    0&  a& c&0\\
    0&b&1/a& 0\\
    0&0&0&d
\end{pmatrix},~~a\neq0\neq d,bc\neq 1$   \\ \hline
\end{longtable}
}\end{center}

\begin{rem}{\rm
It seems that our method in previous sections might be applied to discuss the Lie group structures of the automorphism groups in Table \ref{4a}.
} \end{rem}

\section{\scshape Representations of $\om$-Lie Algebras}

\setcounter{equation}{0}
\renewcommand{\theequation}
{6.\arabic{equation}}
\setcounter{theorem}{0}
\renewcommand{\thetheorem}
{6.\arabic{theorem}}

\begin{defn}
Let $\g$ be a finite-dimensional $\om$-Lie algebra over $\C$ and $V$ a finite-dimensional vector space over $\C$.
We say that $V$ is a \textit{$\g$-module} if there exists a bilinear map $\g\times V\ra V, (x,v)\mapsto x\cdot v$ such that 
\begin{equation}
\label{eq6.1}
[x,y]\cdot v= x\cdot(y\cdot v)-y\cdot(x\cdot v)+\om(x,y)v
\end{equation}
for all $x,y\in\g$ and $v\in V.$ 
\end{defn}
For example, let $\g=B$ denote the complex 3-dimensional $\om$-Lie algebra in Theorem 
\ref{t3d} and $V=\C$.
By Propositions \ref{prop6.2} and \ref{prop6.4} below, we see that $\C$ is an one-dimensional $B$-module. 

Note that if $x \cdot v=0$ for all $x\in\g$ and $v\in V$, then 
it does not make $V$ to be a $\g$-module, unless $\g$ is a Lie algebra.

\begin{prop}
Let $\g$ be a finite-dimensional $\om$-Lie algebra and $V$ a $\g$-module. Consider the direct sum $\g\oplus V$ of vector spaces. Define a skew-symmetric bilinear bracket operation 
$[-,-]:(\g\oplus V)\times (\g\oplus V)\ra\g\oplus V$ by
$$[(x,v),(y,u)]:=([x,y],x\cdot u-y\cdot v)$$
and  define a bilinear form $\Omega: (\g\oplus V)\times(\g\oplus V)\ra\C$ by
$$\Omega((x,v),(y,u)):=\om(x,y)$$ for any $x,y\in\g$ and $u,v\in V$. Then
$(\g\oplus V,[-,-],\Omega)$ is an $\Omega$-Lie algebra, which is called the semi-direct product of $\g$ and $V$.
 \end{prop}

\begin{proof}
For any $x,y,z\in\g$ and $v,u,w\in V$ we have
\begin{eqnarray*}
[[(x,v),(y,u)],(z,w)] & = & [([x,y],x\cdot u-y\cdot v),(z,w)]  \\
 & = & ([[x,y],z],[x,y]\cdot w-z\cdot(x\cdot u-y\cdot v))\\
 & = & ([[x,y],z],x\cdot (y\cdot w)-y\cdot (x\cdot w)+\om(x,y)w-z\cdot (x\cdot u)+z\cdot (y\cdot v)),\\
~ [[(y,u),(z,w)],(x,v)] & = & [([y,z],y\cdot w-z\cdot u),(x,v)]  \\
 & = & ([[y,z],x],y\cdot (z\cdot v)-z\cdot (y\cdot v)+\om(y,z)v-x\cdot (y\cdot w)+x\cdot (z\cdot u)),\\
~[[(z,w),(x,v)],(y,u)] & = & [([z,x],z\cdot v-x\cdot w),(y,u)]  \\
 & = & ([[z,x],y],z\cdot (x\cdot u)-x\cdot (z\cdot u)+\om(z,x)u-y\cdot (z\cdot v)+y\cdot (x\cdot w)).
\end{eqnarray*}
Thus 
\begin{eqnarray*}
&  &  [[(x,v),(y,u)],(z,w)] + [[(y,u),(z,w)],(x,v)] + [[(z,w),(x,v)],(y,u)]\\
& = & ([[x,y],z]+[[y,z],x]+[[z,x],y], \om(x,y)w+\om(y,z)v+\om(z,x)u)\\
& = & (\om(x,y)z+\om(y,z)x+\om(z,x)y, \om(x,y)w+\om(y,z)v+\om(z,x)u)\\
&=&\om(y,z)(x,v)+\om(z,x)(y,u)+\om(x,y)(z,w)\\
 &=&\Omega((y,u),(z,w))(x,v)+\Omega((z,w),(x,v))(y,u)+\Omega((x,v),(y,u))(z,w).
\end{eqnarray*}
This shows that $(\g\oplus V,[-,-],\Omega)$ is an $\Omega$-Lie algebra.
\end{proof}

\begin{prop}\label{prop6.2}
Let $\g$ be a finite-dimensional nontrivial  $\om$-Lie algebra over $\C$. Then  $\C$ is a 1-dimensional $\g$-module if and only if there exists
a linear functional $\tau\in \g^{*}$ such that $\om(x,y)=\tau([x,y])$ for all $x,y\in\g$.
 \end{prop}

\begin{proof}
$(\RA)$ Suppose $c_{0}\in\C$ is a nonzero complex number. Since $\C$ is a 1-dimensional $\g$-module,  then for any $x\in \g$, we can assume that $x\cdot c_{0}=\tau(x)\cdot c_{0}$ for some function $\tau:\g\ra\C$. Since the action of $\g$ on $\C$ is bilinear,  $\tau\in \g^{*}$ is a linear functional. Thus for any $y\in\g$, we have
\begin{eqnarray*}
\tau([x,y])\cdot c_{0}&=&[x,y]\cdot c_{0}\\
&=&x\cdot(y\cdot c_{0})-y\cdot(x\cdot c_{0})+\om(x,y)\cdot c_{0}\\
&=& \tau(x)\cdot (\tau(y)\cdot c_{0})-\tau(y)\cdot(\tau(x)\cdot c_{0})+\om(x,y)\cdot c_{0}\\
&=&\om(x,y)\cdot c_{0}.
\end{eqnarray*}
Hence,  $\om(x,y)=\tau([x,y])$  for all $x,y\in\g$.

$(\LA)$ Define a map $\rhoup:\g\times \C\ra \C, (x,c)\mapsto x \cdot c$ by $x\cdot c:=\tau(x)\cdot c$ for any $x\in \g$ and $c\in\C$. Clearly $\rhoup$ is bilinear. Moreover,
\begin{eqnarray*}\label{ }
[x,y]\cdot c&=&\tau([x,y])\cdot c\\
&=&\om(x,y)\cdot c\\
&=&\om(x,y)\cdot c+ \tau(x)\cdot (\tau(y)\cdot c)-\tau(y)\cdot(\tau(x)\cdot c)\\
&=& x\cdot(y\cdot c)-y\cdot(x\cdot c)+\om(x,y)\cdot c,
\end{eqnarray*}
which means that $\rhoup$ makes $\C$ to be a $\g$-module. 
\end{proof}

\begin{rem}{\rm
This result is essentially due to Zusmanovich \cite[Section 2]{Zus2010}. An $\om$-Lie algebra having a 1-dimensional module 
is called \textit{multiplicative}.
} \end{rem}

\begin{thm}\label{prop6.4}
All 3-dimensional nontrivial  $\om$-Lie algebras over $\C$ are multiplicative.
 \end{thm}

\begin{proof}
According to the classification of 3-dimensional $\om$-Lie algebras (Theorem \ref{t3d}),
we need only to show that every $\om$-Lie algebra $\g$ in the list is multiplicative. By Proposition \ref{prop6.2}, it suffices to find a
 linear functional $\tau\in \g^{*}$ such that $\om(x,y)=\tau([x,y])$ for all $x,y\in\g$. In fact, suppose  $\{x^{*},y^{*},z^{*}\}$ is a dual basis for $\g^{*}$. For $L_{1}$ and $L_{2}$, such $\tau$ could be $y^{*}$. For $A_{\alphaup}$, we can take $\tau=-z^{*}$. 
We can take $\tau=2x^{*}$ for $B$, and $\tau=(1+\alphaup)x^{*}$ for $C_{\alphaup}$.
\end{proof}

\begin{exam}\label{exam6.5}
{\rm
The 4-dimensional $\om$-Lie algebra $L_{1,1}$ is not multiplicative. Assume by the way of contradiction that $L_{1,1}$ is multiplicative, then it follows from Proposition \ref{prop6.2} that there is a linear functional $\tau\in L_{1,1}^{*}$ such that $\om(x,y)=\tau([x,y])$ for all $x,y\in L_{1,1}$. Suppose $\{x,y,z,e\}$ is a basis for $L_{1,1}$. By the generating  relations in $L_{1,1}$ (see Table \ref{4d} previously), we have $\tau(y)=\tau([x,y])=\om(x,y)=1$ and $-\tau(y)=\tau(-y)=\tau([e,y])=\om(e,y)=0$, contradiction. 
} \end{exam}

\begin{rem}{\rm
We can use the method in Proposition \ref{prop6.4} and Example \ref{exam6.5} to classify all 4-dimensional multiplicative $\om$-Lie algebras over $\C$.
} \end{rem}

\begin{rem}[Zusmanovich \cite{Zus2010}]{\rm By the definition of $\g$-module, one can check that
an $\om$-Lie algebra $\g$ is a module over itself under the adjoint action if and only if $\g$ is trivial, i.e., $\g$ is a Lie algebra.
}\end{rem}

\section{\scshape Irreducible Representations of $C_{\alphaup}$}

\setcounter{equation}{0}
\renewcommand{\theequation}
{7.\arabic{equation}}
\setcounter{theorem}{0}
\renewcommand{\thetheorem}
{7.\arabic{theorem}}

In this section we study the irreducible representations of $C_{\alphaup}$. Theorem \ref{thm7.1} is our main result in this section, showing that any irreducible $C_{\alphaup}$-module  is 1-dimensional, whenever $0,-1\neq\alphaup\in \mathbb{C}.$

\begin{thm}\label{thm7.1}
Any irreducible $C_{\alphaup}$-module  is 1-dimensional, where $\alphaup\neq 0,-1$. 
 \end{thm}
 
 \begin{proof}
Recall that $C_{\alphaup} :  [x,y]=y, [x,z]=\alphaup z, [y,z]=x$ with $\om(x,y)=\om(x,z)=0,\om(y,z)=1+\alphaup$, where  $0,-1\neq\alphaup\in \mathbb{C}.$ 
Suppose that $V$ is a finite-dimensional  irreducible  $C_{\alphaup}$-module. We may view every element in $C_{\alphaup}$ as a linear map from 
$V$ to itself. Since the linear map $x$ is over $\C$, it must have an eigenvector $v_{0}\neq 0$, say.  
Suppose $\eta$ is the corresponding eigenvalue, then $x\cdot v_{0}=\eta v_{0}$.

It follows from Equation (\ref{eq6.1}) that $x\cdot(y\cdot v_{0})=[x,y]\cdot v_{0}+y\cdot(x\cdot v_{0})-\om(x,y)v_{0}=y\cdot v_{0}+y\cdot(\eta v_{0})=(\eta+1)(y\cdot v_{0})$.
Thus $y\cdot v_{0}$ is also an eigenvector of $x$ with eigenvalue $\eta+1$.
Furthermore, we observe that 
\begin{equation}\label{eq7.1}
x\cdot (y^{k} \cdot v_{0})=(\eta+k)(y^{k} \cdot v_{0})
\end{equation}
In fact, 
\begin{eqnarray*}
x\cdot (y^{k} \cdot v_{0})&=&x\cdot (y \cdot (y^{k-1}\cdot v_{0}))\\
&=&[x,y]\cdot (y^{k-1}\cdot v_{0})+y\cdot(x\cdot (y^{k-1}\cdot v_{0}))-\om(x,y)\cdot y^{k-1}\cdot v_{0} \quad\textrm{ (by Equation } (\ref{eq6.1}))\\
&=&y\cdot (y^{k-1}\cdot v_{0})+y\cdot(x\cdot (y^{k-1}\cdot v_{0}))\\
&=&y^{k}\cdot v_{0}+y\cdot((\eta+k-1)(y^{k-1}\cdot v_{0})) \quad\textrm{ (by induction hyperthesis) } \\
&=& (\eta+k)(y^{k} \cdot v_{0}). 
\end{eqnarray*}
Thus Equation (\ref{eq7.1}) holds. This means that the vectors $\{y^{k}\cdot v_{0}\mid k=0,1,2,3,\dots\}$
are either zero or eigenvectors of $x$. Note that any two eigenvalues in $\{\eta+k\mid k=0,1,2,3,\dots\}$ are distinct, and since
the eigenvectors corresponding to different eigenvalues are linear independent, we may find a minimal integer $k_{0}$ such that $y^{k_{0}}\cdot v_{0}\neq 0$ and $y^{k_{0}+1}\cdot v_{0}=0$.
Let $v_{1}:=y^{k_{0}}\cdot v_{0}$ and $\eta_{1}:=\eta+k_{0}$, we have seen that
$$x\cdot v_{1}=\eta_{1} v_{1}\textrm{ and } y\cdot v_{1}=0.$$

In order to analyze $x\cdot(z^{r}\cdot v_{1})$, we use the induction on $r$ to see that
\begin{eqnarray*}
x\cdot(z^{r}\cdot v_{1})&=&[x,z]\cdot (z^{r-1}\cdot v_{1})+z\cdot(x\cdot  (z^{r-1}\cdot v_{1}))-\om(x,z) (z^{r-1}\cdot v_{1})\\
&=&\alphaup(z\cdot  (z^{r-1}\cdot v_{1}))+z\cdot((\eta_{1}+(r-1)\alphaup) z^{r-1}\cdot v_{1})\\
&=&(r\alphaup+\eta_{1}) (z^{r}\cdot v_{1}).
\end{eqnarray*}
Namely, $\{z^{r}\cdot v_{1}\mid r=0,1,2,3,\dots\}$ are also  eigenvectors of $x$. 
Since $\alphaup\neq 0$, any two eigenvalues in $\{r\alphaup+\eta_{1}\mid r=0,1,2,3,\dots\}$ are distinct.
Thus we may find a minimal integer $n$ such that 
$$z^{n}\cdot v_{1}\neq 0\textrm{ and }z^{n+1}\cdot v_{1}=0.$$

Let $V_{0}$ be the subspace generated by $\{z^{j}\cdot v_{1}\mid 0\leqslant j\leqslant n\}$, then $V_{0}\neq \{0\}$.
To see that $V_{0}$ is a $C_{\alphaup}$-submodule of $V$, it suffices to show that 
$V_{0}$ is stable under the actions of $x,y,$ and $z$. Clearly, $x(V_{0})\subseteq V_{0}$ and 
$z(V_{0})\subseteq V_{0}$. Moreover, we claim that 
\begin{equation}
\label{eq7.2}
y\cdot(z^{j}\cdot v_{1})=j\cdot\Big(\eta_{1}-1+\frac{(j-3)\alphaup}{2}\Big)  (z^{j-1}\cdot v_{1})
\end{equation}
for all $j>0$. In fact, the induction on $j$ shows that
\begin{eqnarray*}
y\cdot(z^{j}\cdot v_{1})& = & y\cdot(z\cdot (z^{j-1}\cdot v_{1})) \\
 & = & [y,z]\cdot(z^{j-1}\cdot v_{1})+z\cdot(y\cdot (z^{j-1}\cdot v_{1}))-\om(y,z)(z^{j-1}\cdot v_{1})\\
 &=& x\cdot(z^{j-1}\cdot v_{1})+z\cdot(y\cdot (z^{j-1}\cdot v_{1}))-(\alphaup+1)(z^{j-1}\cdot v_{1})\\
 &=&\bigg(\eta_{1}+\alphaup (j-1)+(j-1)\cdot\Big(\eta_{1}-1+\frac{(j-4)\alphaup}{2}\Big)-(\alphaup+1)\bigg)(z^{j-1}\cdot v_{1})\\
 &=&j\cdot\bigg(\eta_{1}-1+\frac{(j-3)\alphaup}{2}\bigg)  (z^{j-1}\cdot v_{1}).
\end{eqnarray*}
Thus, this claim holds. Equation (\ref{eq7.2}), together with $y\cdot v_{1}=0$ implies that $y(V_{0})\subseteq V_{0}$. Hence, $V_{0}$ is a nonzero  $C_{\alphaup}$-submodule of $V$. Since $V$ is irreducible, we have $V=V_{0}$ and $\dim(V)=\dim(V_{0})=n+1$.

Now we want to seek when the irreducible representation $V$ exists. We are going to show that there are no irreducible representations in dimension $>1$. We still follow the previous notations, and suppose
$V=\langle v_{1},z^{1}\cdot v_{1},\dots,z^{n}\cdot v_{1}\rangle$ is an irreducible $C_{\alphaup}$-module of dimension $n+1$.

Taking $j=n+1$ in Equation (\ref{eq7.2}), we obtain
$$(n+1)\cdot\Big(\eta_{1}-1+\frac{(n-2)\alphaup}{2}\Big)=0$$
which, together with the fact that $n\in\N$ implies 
\begin{equation}
\label{eq7.3}
\eta_{1}=1-\frac{(n-2)\alphaup}{2}.
\end{equation}

To prove that $n=0$, we assume by the way of contradiction that $n>0.$ Then $z\cdot v_{1}\neq 0$ in $V$, and
\begin{eqnarray*}
[x,y]\cdot(z\cdot v_{1}) & = & x\cdot(y\cdot(z\cdot v_{1}))-y\cdot(x\cdot(z\cdot v_{1}))+\om(x,y)(z\cdot v_{1}) \\
 & = &\Big(\eta_{1}-1-\alphaup\Big)  (x\cdot v_{1})-(\eta_{1}+\alphaup) (y\cdot(z\cdot v_{1}))\\
 &=&\Big(\eta_{1}-1-\alphaup\Big) \eta_{1} \cdot v_{1}-(\eta_{1}+\alphaup) \Big(\eta_{1}-1+\alphaup\Big)\cdot v_{1}\\
 &=&-\alphaup\Big(\eta_{1}-1-\alphaup\Big) \cdot v_{1}
\end{eqnarray*}
Since $[x,y]=y$ in $C_{\alphaup}$, then $[x,y]\cdot v=y\cdot v$ for any element in  $V$. Thus 
$$y\cdot(z\cdot v_{1})=[x,y]\cdot(z\cdot v_{1})$$ gives rise to $\eta_{1}-1-\alphaup =  -\alphaup(\eta_{1}-1-\alphaup) 
$, i.e.,
\begin{equation}
 (1+\alphaup)(\eta_{1}-1-\alphaup) =0.
\end{equation}
Since $\alphaup\neq -1$, it follows that $$\eta_{1}-1-\alphaup=0$$
which, combining with Equation (\ref{eq7.3}) obtains
$$\alphaup=-\frac{(n-2)\alpha}{2}.$$
Recall that $\alphaup\neq 0$. Therefore, $n=0$. 
The proof is completed. 
 \end{proof}

\begin{rem}{\rm
If we extend the parameter $\alphaup$ in $C_{\alphaup}$ to contain the case of $\alphaup=-1$, then it is easy to see that $C_{-1}=\mathfrak{sl}_{2}(\C)$, the special linear Lie algebra, which is trivial as an $\om$-Lie algebra. It is well-known that up to equivalence,  $C_{-1}$ has one unique irreducible representation in any dimension $n<\infty$, see for example, Fulton-Harris \cite[Section 11.1]{FH1991}. This result, comparing with 
Theorem \ref{thm7.1} shows that representation theory of nontrivial $\om$-Lie algebras might be very ``degenerate'' .
} \end{rem}

Finally, we give some remarks for more potential research in the theory of $\om$-Lie algebras.

\begin{rem}{\rm
Note that our many results might be valid over other fields of characteristic zero, especially over the field 
$\mathbb{R}$ of real numbers. Real $\om$-Lie algebras should deserve more attention  because the motivation and the first application of introducing $\om$-Lie algebras  were presented over $\mathbb{R}$
in mathematical physics, see Nurowski \cite{Nur2007}. Based on the links between physics and low-dimensional 
nonassociative algebras have been found, for example, see Zhang-Bai \cite{ZB2012} and reference therein, we also expect that 
our computations might have more  applications in mathematical physics. 

In terms of pure mathematics, there are several possible directions concerning the theory of $\om$-Lie algebras that might be worth 
investigating. The first one was to develop  a theory of infinite dimensional  $\om$-Lie algebras. In the last section of 
 Chen-Zhang \cite{CZ2016}, we have already constructed an example of nontrivial infinite dimensional $\om$-Lie algebras. But it seems that 
 there are no more results even examples appeared so far. The second one was the study of solvable and nilpotent 
 $\om$-Lie algebras. In Chen-Liu-Zhang \cite{CLZ2014} and \cite{CZ2016}, we have completed the classification of complex simple finite-dimensional $\om$-Lie algebras. It is natural to consider the structures of solvable and nilpotent 
 $\om$-Lie algebras. In this paper we did not discuss the solvability and nilpotence of $\om$-Lie algebras but according to 
 the experience we gained, the present calculations probably provide a concrete guide for more research in this direction. 
 The third direction was the restricted version of $\om$-Lie algebras, i.e., $\om$-Lie algebra over a field of characteristic $p>0$ 
 satisfying some restricted conditions in sense of Jacobson \cite[Chapter V, Section 7]{Jac1979}.
 However, we do not know how to give a correct definition of restricted $\om$-Lie algebras and whether there is a 
 big difference between modular $\om$-Lie algebras and modular Lie algebras.
} \end{rem}

\section*{\textit{Acknowledgments}}

This research was partially supported by NNSF of China (No. 11301061, 11401087) and Ministry of Educations of China (No. 201510200029). We would like to thank the referee for careful reading and useful comments.


\end{document}